\documentclass{amsart}
\usepackage{amsfonts,amssymb,amsmath,amsthm}
\usepackage{url}
\usepackage{enumerate}
\usepackage{fullpage}
\usepackage{bm}
\usepackage{wasysym}
\usepackage[usenames,dvipsnames]{color}
\usepackage{amscd}

\newtheorem{thm}{Theorem}[section]
\newtheorem{lem}[thm]{Lemma}
\newtheorem{prop}[thm]{Proposition}
\newtheorem{cor}[thm]{Corollary}
\theoremstyle{definition}
\newtheorem{defn}[thm]{Definition}
\newtheorem{remk}[thm]{Remark}
\newtheorem{exm}[thm]{Example}
\newtheorem{conj}[thm]{Conjecture}%[section]
\numberwithin{equation}{section}

\newcommand{\Aff}{\mathbb{A}}      % Affine space
\newcommand{\C}{\mathbb{C}}

\newcommand{\R}{\mathbb{R}}
\newcommand{\Z}{\mathbb{Z}}
\newcommand{\Q}{\mathbb{Q}}
\newcommand{\PP}{\mathbb{P}}
\newcommand{\N}{\mathbb{N}}
\newcommand{\TT}{\mathbb{T}}%Torus
\newcommand{\G}{\mathbb{G}}  % blackboard bold G for the multiplicative\additive group
\newcommand{\m}{\mathrm{m}}

\newcommand{\J}{\mathcal J}
\newcommand{\K}{\mathcal K}
\newcommand{\ord}{\mathrm{ord}}
\newcommand{\id}{\mathrm{id}}

\newcommand{\Aut}{\operatorname{Aut}}
\newcommand{\End}{\operatorname{End}}
\newcommand{\tth}{^{\operatorname{th}}}

\author{Annie Carter}
\author{Matilde Lal\'in}
\author{Michelle Manes}
\author{Alison Beth Miller} 
\author{Lucia Mocz}

\address{Annie Carter: Department of Mathematics, University of California San Diego, 9500 Gilman Drive \# 0112, La Jolla, CA  92093-0112, USA}\email{a4carter@ucsd.edu}

\address{Matilde Lal\'in:  D\'epartement de math\'ematiques et de statistique,
                                    Universit\'e de Montr\'eal.
                                    CP 6128, succ. Centre-ville.
                                     Montreal, QC H3C 3J7, Canada}\email{mlalin@dms.umontreal.ca}

\address{Michelle Manes: Department of Mathematics, University of Hawaii, 2565 McCarthy Mall, Honolulu, HI 96822, USA}\email{mmanes@math.hawaii.edu}

\address{Alison Beth Miller: Mathematical Reviews, 416 Fourth St., Ann Arbor, MI 48103, USA}\email{alimil@umich.edu }

\address{Lucia Mocz: Department of Mathematics,
University of Chicago, 5734 S. University Avenue, Room 108
Chicago, IL, 60637, USA}\email{lmocz@math.uchicago.edu }

\title{Two-variable polynomials with dynamical Mahler measure zero}

\subjclass[2020]{Primary 11R06; Secondary 11G50, 37P15, 37P30}

\keywords{Mahler measure, dynamical Mahler measure, polynomial, preperiodic points, equidistribution}

\begin{document}

	%%%%%%%%%%%%%%%%%%%%%%%%%%%%%%%%%%%%%%%%%%%%%%%%%%%%%%%%%%%%%%%%%%%%%%%%%%%%
	\begin{abstract}
We discuss several aspects of the dynamical Mahler measure for multivariate polynomials. We prove a weak dynamical version of Boyd--Lawton formula and we characterize the polynomials with integer coefficients having dynamical Mahler measure zero both for the case of one variable (Kronecker's lemma) and for the case of two variables, under the assumption that the dynamical version of Lehmer's question is true. 
	\end{abstract}
	%%%%%%%%%%%%%%%%%%%%%%%%%%%%%%%%%%%%%%%%%%%%%%%%%%%%%%%%%%%%%%%%%%%%%%%%%%%%

	\maketitle
	
	% if using multiple files, insert command below at the top of every auxiliary file.
	% \ifx
	% 	\thepage\undefined\def\jobname{MainFileName}
	% 	\input{MainFileName}
	% \fi

\section{Introduction}

This paper considers a generalized version of Mahler measure attached to a discrete dynamical system. 
In~\cite{Szpiro-Tucker, PST}, the authors consider a single-variable generalized Mahler measure, which we extend here to multivariate polynomials.
We begin with some background on both Mahler measure and discrete dynamical systems as well as the connections between them.

\subsection{Mahler measure}
The (logarithmic) {\bf Mahler measure} of a non-zero rational function $P \in \C(x_1, \dots
x_n)^\times$ is defined by
\[\m(P) = \frac{1}{(2\pi i)^n} \int_{\TT^n} \log | P(z_1, \cdots, z_n)| \frac{d
z_1}{z_1} 
\cdots \frac{d z_n}{z_n},\]
where $\TT^n = \{ (z_1, \dots, z_n) \in \C^n \, :\, |z_1| = \cdots = |z_n| = 1 \}$ is
the unit torus.  

In the particular case of a single variable polynomial, Jensen's formula implies that if
\[P(x) = a \prod_i(x-\alpha_i),\]
then
\begin{equation}\label{eq:onevariable}
\m(P)= \log |a| + \sum_{|\alpha_i|>1} \log |\alpha_i|.
\end{equation}

The single variable polynomial version  was first introduced by Lehmer \cite{Le} as a useful tool in a method for producing large prime numbers by generalizing Mersenne's sequences. 

It is natural to ask the question about the range of possible values of Mahler measure. For $P\in \Z[x]$, Kronecker's lemma \cite{Kronecker} implies that $\m(P)=0$ if and only if $P$ is (up to a sign) monic and a product of cyclotomic polynomials and a monomial. 

From the context of finding large prime numbers, Lehmer posed the following question:\\
\begin{quotation}
{\em Given $\varepsilon >0$, can we find a polynomial $P \in \Z[x]$ such that $0<\m(P)<\varepsilon$?}\\
\end{quotation}
Lehmer showed that
\[ \m(x^{10}+x^9-x^7-x^6-x^5-x^4-x^3+x+1 ) =  0.162357612\dots\]
and declared that this was the polynomial with the smallest positive measure that he was able to find. 
To this day, Lehmer's question remains unanswered and Lehmer's polynomial remains the one with the smallest known positive Mahler measure.
Some key progress on this question can be found, for example, in the works of Breusch \cite{Breusch} and Smyth \cite{Smyth-nonreciprocal}, who proved that the Mahler measures of non-reciprocal polynomials have a lower bound;  Dobrowolski \cite{Dobrowolski}, who gave the first lower bound that approaches zero when the degree goes to infinity; and the recent proof of Dimitrov \cite{Dimitrov} of the Schinzel--Zassenhaus conjecture. 

 Mahler measure for multivariate polynomials was first considered by Mahler \cite{Mah} in connection to heights and their applications in transcendence theory. 
Later Smyth \cite{Smyth1, Boyd-speculations} obtained the first exact formulas involving Dirichlet $L$-functions and the Riemann zeta function, such as 
\begin{align*}
\m(1+x+y) =& \frac{3 \sqrt{3}}{4 \pi} L(\chi_{-3},2),\\
\m(1+x+y+z) =& \frac{7}{2\pi^2}\zeta(3).
\end{align*}
More recently, various formulas have been proven or conjectured involving $L$-functions associated to more complicated arithmetic-geometric objects such as elliptic curves. For example, the following formula was conjectured by Deninger  \cite{Deninger} and Boyd \cite{Bo98}  and proven by Rogers and Zudilin \cite{RZ14}
\[\m\left(x +\frac{1}{x}+y+\frac{1}{y}+1\right)=\frac{15}{4 \pi^2}L(E_{15},2),\]
where $E_{15}$ is an elliptic curve of conductor 15. 

We finish the discussion on classical Mahler measure by recalling a result due to Boyd \cite{Boyd-speculations} and Lawton \cite{Lawton}. If $P\in \C(x_1,\dots,x_n)^\times$, then 
\begin{equation}\label{eq:BL}
\lim_{q({\bf k})\rightarrow \infty} \m (P(x,x^{k_2},\dots,x^{k_n}))=\m(P(x_1,\dots,x_n)),
\end{equation}
where 
\[q({\bf k})=\min \left\{ H({\bf s}) : {\bf s}=(s_2,\dots,s_n) \in \Z^{n-1}, {\bf s} \not = (0,\dots,0), \mbox{ and }\sum_{j=2}^n s_j k_j =0\right\}\]
where $H({\bf s})=\max\{|s_j|: 2\leq j \leq n\}$.  Intituively, the above limit is taken while $k_2,\dots,k_n$ go to infinity independently from each other. 

This result expresses the Mahler measure of a multivariate polynomial in terms of the Mahler measure of a single variable polynomial. It sheds light into the Lehmer conjecture for multivariate polynomials, since the existence of any such polynomial with small Mahler measure would immediately imply the existence of infinitely many single variable polynomials with small Mahler measure.

\subsection{Arithmetic dynamics}
In classical (holomorphic) dynamics, one studies topological and analytic
properties of orbits of points under iteration of self-maps  of $\C$ or, more generally, of a complex manifold. In the younger field of arithmetic dynamics, the self-maps act on number theoretic objects (for example, algebraic varieties) and the objects and maps are defined over fields of number theoretic interest 
(number fields, function fields, finite fields, non-archimedean fields, etc.). 

We fix the following notation: Let $X$ be a set, possibly with additional structure. 
For $f: X \to X$  and  $L \in \Aut(X)$, we write 
\begin{equation}\label{def:conj-notation}
f^n = \underbrace{f\circ f \circ \cdots \circ f}_{\text{$n$-fold composition}},
\quad\text{ and }\quad f^L := L^{-1}\circ f\circ L.
\end{equation}
This conjugation is a natural dynamical equivalence relation in that it respects iteration; that is, $(f^L)^n = (f^n)^L$.  In the sequel, we are concerned with polynomial maps $f: \C \to \C$, so we consider equivalence up to conjugation by affine maps $L(z) = az+b \in \C[z]$ with $a\neq 0$.

Let $\alpha \in X$. The forward orbit of $\alpha$ under $f$ is the set $\mathcal{O}_f(\alpha) = \{ f^n(\alpha) \, :\, n \geq 0 \}$.
 Questions in arithmetic dynamics are often motivated by an analogy between arithmetic geometry and dynamical systems in which, for example, rational and integral points on varieties correspond to rational and integral points in orbits, and torsion points on abelian varieties correspond to preperiodic points (points with finite orbit). See~\cite{Silverman-arithmetic-dynamical} for more comprehensive background and motivation and~\cite{CurrentTrends} for a survey of recent progress in the field of arithmetic dynamics.

Two complementary sets play an important role in holomorphic dynamics: The {\bf Fatou set}  of $f$ is the maximal open set on which the family of iterates  $\{f^n \, :\, n\geq 1\}$ is 
equicontinuous (considered as maps on $\PP^1(\C)$); its complement is the {\bf Julia set} of $f$, which we denote $\J_f$. Informally, the Julia set is where the interesting dynamics happens.
For a polynomial $f: \C \to \C$, we may define the Julia set a bit more concretely.

Let $f\in \C[z]$. The {\bf filled Julia set} of $f$ is
 \[\K_f=\{z\in \C \, :\, f^n(z)\not \rightarrow \infty \mbox{ as } n\rightarrow \infty\}.\]
 The {\bf Julia set} of $f$ is the boundary of the filled Julia set. That is, $\J_f = \partial\K_f$.
It follows from these definitions that for a polynomial $f\in \C[z]$, both $\K_f$ and $\J_f$ are compact.
In dynamical Mahler measure, the Julia set $\J_f$ for a general polynomial $f$ will play the role of the unit torus $\TT$.

Two families of polynomials play an important role in this work: The pure power maps $f(z)=z^d$ and the Chebyshev polynomials. Both of these families have special dynamical properties that arise because they are related to endomorphisms of algebraic groups.

Consider the multiplicative group $\G_m$ where for a field $K$,  the $K$-valued points are $\G_m(K) = K^*$.  The endomorphism ring of $\G_m$ is $\Z$:
\begin{align*}
\Z &\to \End(\G_m)\\
d &\mapsto z^d.
\end{align*}
 Iteration of pure power maps is particularly easy to understand:
\[
f(z) = z^d \quad \text{ means } \quad f^n(z) = z^{d^n}.
\]
For $d\geq 2$ and $K\subseteq \C$ we have three cases: If $|\alpha| >1$, then $|\alpha^{d^n}|  \to \infty $ with $n \to \infty$.
If $|\alpha| <1$, then $|\alpha^{d^n}|  \to 0  $ with $n \to \infty$.
If $|\alpha| =1$, then $|\alpha^{d^n}|  = 1 $ for all $n$.
So for  pure power maps, we can understand the Julia sets completely: $\K_f$ is the unit disc, and $\J_f$ is the unit circle.

To give another example, since the automorphism of $\G_m$ given by $z\mapsto z^{-1}$ commutes with the power map $z \mapsto z^d$, the power map descends to an endomorphism of $\Aff^1$, which we denote $T_d$, the $d\tth$ Chebyshev polynomial.\footnote{Classically the Chebyshev polynomials are normalized so that $\widetilde{T}_d(\cos t) = \cos(dt)$.  The two normalizations satisfy $\widetilde{T}_d(w) = \frac 1 2 T_d (2w)$. }

\def\zd{{z\mapsto z^d}}\def\Td{{w \mapsto T_d(w)}} \def\zinverse{{z \mapsto z+z^{-1}}}
\begin{equation*}
\begin{CD}
\G_m @>\zd>> \G_m \\
@VVV @VVV\\
\G_m /\{z = z^{-1} \}@>\zd>> \G_m / \{z = z^{-1}\} \\
@VV\zinverse V @VV\zinverse V\\
\Aff^1 @>\Td>> \Aff^1 \\
\end{CD}
\end{equation*}

Taking as a definition the fact that $T_d \in \Z[w]$ satisfies
\begin{equation}
T_d(z+z^{-1}) = z^d + z^{-d},
\label{eqn:chebyshev def}
\end{equation}
one may prove existence and uniqueness of the Chebyshev polynomials along with a recursive formula:
\begin{equation}
T_d(w) = 
\begin{cases}
2 & d = 0,\\
w & d = 1,\\
w T_{d-1}(w) - T_{d-2}(w) & d \geq 2.
\end{cases}\label{eqn: Chebyshev}
\end{equation}
The following rule for composition of Chebyshev polynomials follows directly from the definition in~\eqref{eqn:chebyshev def}:
\[
T_d \circ T_e(w) = T_{de}(w) = T_e \circ T_d(w),
\] 
and it gives a simple form of iteration
\begin{equation}\label{eqn: cheby iter}
T_d^n(w) = T_{d^n}(w).
\end{equation}

For $d\geq 2$, we have two cases for $T_d:\C \rightarrow \C$: If $\alpha \in [-2, 2]$, then $T_d(\alpha) \in [-2, 2]$. If $\alpha \in \C\setminus [-2,2]$, then $|T_d^n(\alpha)| = |T_{d^n}(\alpha)| \to \infty$ with $n \to \infty$. Thus we understand the Julia set in this case as well: $\K_{T_d} = \J_{T_d} = [-2,2]$.

See~\cite[Chapter 6]{Silverman-arithmetic-dynamical} for more on the dynamics of pure power maps, Chebyshev polynomials, and other  maps arising from algebraic groups, including proofs of some of the statements above.

\subsection{Dynamical Mahler measure and statement of results}
An important tool in the study of arithmetic dynamics is the canonical height of a point in $ \PP^1$ relative to a morphism $f$, defined by Call and Silverman \cite{Call-Silverman} and modeled after Tate's construction of the canonical
(N\'eron-Tate) height on abelian varieties.
The standard (Weil) height of a point $\beta = \frac a b \in \Q$ (written in lowest terms) can be computed as $h(\beta) = \max\{ |a|, |b| \}$, and, observing that $a$ and $b$ are (up to sign) the coefficients of the minimal polynomial of $\alpha$ over $\Z$, this definition extends easily to algebraic numbers $\alpha \in \bar \Q$. 
If $f \in \Q[z]$ is a polynomial and $\alpha \in \PP^1_{\bar \Q}$, we then define:
\[
\hat h_f(\alpha) = \lim_{n \to \infty} d^{-n} h(f^n(\alpha)),
\]
where $f^n$ represents the $n$-fold composition of $f$ and $d=\deg(f)$. In analogy with the canonical height on abelian varieties, we find that this limit exists and that
\begin{itemize}
\item
$\hat h_f(f(\alpha)) = (\deg f) \hat h_f(\alpha)$, and
\item
$\hat  h_f(\alpha) = 0$ if and only if $\alpha$ is a preperiodic point for $f$ (that is, $f^m(\alpha) = f^n(\alpha) $ for some natural numbers $m>n$).
\end{itemize}

Let $\alpha \in \bar \Q$ with minimal polynomial $P \in \Z[z]$. The theory of heights and Mahler measure are connected via the formula
\begin{equation}\label{eqn:stdmahler}
[\Q(\alpha): \Q]  h(\alpha) =  \m(P).
\end{equation}

Let $\K\subset \C$ be compact. For any Borel probability measure $\nu$ on $\K$ the energy of $\nu$ is given by~\cite[Definition 3.2.1]{Ransford}:
\[I(\nu)=\int_\K \int_\K \log |z-w| d\nu(z) d\nu(w).\] 
An equilibrium measure on $\K$ is a measure $\mu$ satisfying 
\[I(\mu)=\sup_\nu I(\nu),\]
where the supremum is taken over all Borel probability measures. An equilibrium measure always exists and is often unique \cite[Theorem 3.2.2]{Ransford}. 

Let 
\[f(z)=a_dz^d+\cdots +a_0 \in \C[z],\]
where $a_d \not = 0$ and $d \geq 2$. 

 Brolin \cite{Brolin},  Lyubich~\cite{Lyubich}, and Freire, Lopes, and Ma\~ne~\cite{FLM} construct a unique equilibrium measure $\mu_f$, supported on the Julia set $\J_f$, which is invariant under $f$. That is, $f_*\mu_f=\mu_f$, where the push-forward $f_*\mu$ is defined by $f_*\mu(B)=\mu(f^{-1}(B))$ for any Borel set $B$. We also have
\begin{equation}\label{eq:energycoeff}
I(\mu_f)=-\frac{1}{d-1}\log|a_d|,
\end{equation}
and in particular, $I(\mu_f)=0$ when $f$ is monic.

Suppose that $f \in \Z[z]$ is monic, and that $\beta = \frac a b \in \Q$ (written in lowest terms).  Then one can show that 
\[
\hat h_f(\beta) = \int_{\J_f}\log |bz - a| d\mu_f(z).
\]
More generally, if $\alpha$ is an algebraic integer with monic minimal polynomial $P(z)$, then
\begin{equation}\label{eq:adelic}
[\Q(\alpha): \Q] \hat h_f(\alpha) = \int_{\J_f} \log |P(z)| d\mu_f(z) =: \m_f(P).
\end{equation}
Notice that we are associating  the canonical height for the rational map $f$ to a generalized Mahler measure $\m_f$ for  polynomials $P \in \Z[x]$. Taking $f(z) = z^2$ recovers the familiar Mahler measure from equation~\eqref{eqn:stdmahler}.

Pi\~neiro, Szpiro, and Tucker \cite{PST} have shown that, like heights, this generalized Mahler measure may be computed 
via an adelic formula:
\begin{equation}\label{eq:Kadelic}
[K(\alpha):\Q]\hat h_f(\alpha) = \sum_{\text{places $v$ of }K}
\int_{\J_{f,v}}\log |P(z)|_v d\mu_{f,v}(z),
\end{equation}
where $K$ is a number field, $\alpha \in \PP^1(K)$, $P$ is a minimal polynomial for $\alpha$ over $K$, and 
$\J_{f,v}$ denotes the Julia set of $f: \PP^1(\C_v) \to \PP^1(\C_v)$. 
For finite $v$, $\int_{\J_{f,v}}\log |P(z)|_v d\mu_{f,v}(z)=0$ unless $f$ has bad reduction at $v$ or all the coefficients of $P$ have nonzero $v$-adic valuation. Thus, equation \eqref{eq:adelic} follows from \eqref{eq:Kadelic} by specializing to $K=\Q$ and $\alpha$ algebraic integer.

The contributions at the finite places in equation \eqref{eq:Kadelic} were shown to have the appropriate integral forms by Favre and Rivera-Letelier \cite[Proposition 1.3]{Favre-Rivera-Letelier}. Thus, equation \eqref{eq:adelic} can also be extended to the case where $f$ is not necessarily monic if the integral is replaced by a sum of integrals at different places. 

Szpiro and Tucker \cite{Szpiro-Tucker} used diophantine approximation to
show that this generalized Mahler measure of a polynomial $P$ at a place $v$
 can be computed by averaging $\log |P|_v$ over the periodic points of $f$.

It is not easy to give precise examples of $\m_f(P)$.  In the classical case when $f(z)=z^d$ with $d\geq 2$, then $\m_f(P) = \m(P)$ for $P \in \C(x)^\times$.  For $T_d$ the $d^\text{th}$ Chebyshev polynomial with $d\geq 2$, normalized so that $T_d(z+z^{-1}) = z^d + z^{-d}$, one can show that  $\m_{T_d}(P) = \m\left(P \circ (z + z^{-1})\right)$. Because of the underlying group structure discussed above, these examples are straightforward to compute (see~\cite{CLMM}). For general polynomials, exact computation is much more difficult, and in fact these two families are the exceptions to many theorems stated in the sequel.

One can compute approximations to dynamical Mahler measure in some cases. For example,   Ingram \cite{Ingram} gave an asymptotic formula for $\m_f(x)$ for $f(z)=z^2+c$ as $c \rightarrow \infty$
  \begin{align*}
  \m_f (x) =& \frac{\log |c|}{2}+\frac{1}{8}\log\left(1+\frac{2}{c}+\frac{1}{c^2}+\frac{1}{c^3} \right)+O\left(\frac{1}{c^7}\right)\\
  =&\frac{\log |c|}{2}+\frac{1}{4c}-\frac{1}{8c^2}+\frac{5}{24c^3}- \frac{5}{16c^4}+\frac{17}{40c^5} - \frac{29}{48c^6}+O\left(\frac{1}{c^7}\right).
  \end{align*}
  Similarly for $f(z)=(z-c)^2$, one can find
 \[\m_f(x)=\log|c|-\frac{1}{2c} -\frac{1}{4c^2} - \frac{5}{12c^3} - \frac{5}{8c^4}- \frac{17}{20c^5} - \frac{29}{24c^6}+O\left(\frac{1}{c^7}\right).\]

It is natural to ask about  higher-dimensional (multivariate) forms of this generalized Mahler measure. 

\begin{defn} Let $f\in \Z[z]$ monic of degree $d \ge 2$. The $f$-dynamical Mahler measure of $P \in \C(x_1, \dots x_n)^\times $ is given by 
\begin{equation}\label{eq:dmm-defi}
\m_f(P) = \int_{\J_f} \dots \int_{\J_f} \log | P(z_1, \cdots, z_n)| d\mu_f(z_1) \cdots d\mu_f(z_n).\end{equation}
\end{defn}
We will later see that the above integral always converges and that in fact, $\m_f(P)\geq 0$ if $P \in \Z[x_1, \dots, x_n]$ is nonzero.

It is an interesting to ask whether multivariable dynamical Mahler measure can also be interpreted as a height.  We do not answer the question in this paper, but we suspect that the answer is yes.  Zhang \cite{Zhang}, in the more general setting of a variety $X$ with a map $\Phi: X \to X$ and a line bundle $\mathcal{L}$ such that $\Phi^*\mathcal{L} = \mathcal{L}^d$, defined a canonical height $h_{\Phi, \mathcal{L}}(Y)$ of an arbitrary subvariety $Y \subset X$.  Chambert-Loir and Thuillier \cite{Chambert-Loir-Thuillier} showed that in the case where $X = \PP^n$, $\Phi([x_0, \dotsc x_n]) = [x_0^2, \dotsc, x_n^2] $ and $\mathcal{L} = \mathcal{O}(1)$, that, for a homogeneous polynomial $P \in \Z[x_0, \dotsc, x_n]$ and $Y$ the hypersurface $\{P(x_0, x_1, \dotsc, x_n) = 0\}$, $h_{\Phi, \mathcal{L}}(Y) = \m(P)$ is the usual multivariable Mahler measure of $P$, which is also equal to the multivariable Mahler measure of the inhomogeneous polynomial $P(1, x_1, \dotsc, x_n)$.

We conjecture that the multivariable dynamical Mahler measure can be recovered in a similar way, using instead the map $\Phi = f \times \cdots \times f: (\PP^1)^n \to (\PP^1)^n$, which acts as $f$ on each of the $n$ components.  

In this note, we prove the following results.

\begin{lem}[Dynamical Kronecker's Lemma] \label{lem:dynamicalKronecker} Let $f \in \Z[z]$ be monic 
of degree $d \ge 2$  and let $P \in \Z[x]$. 
If $\m_f(P)=0$,  then $P(x) = \pm \prod_i (x-\alpha_i)$ where each $\alpha_i$ is a preperiodic point of $f$.
\end{lem}

\begin{prop}[Weak Dynamical Boyd--Lawton] \label{prop:weakdynamicalBL}
	Let $f \in \Z[z]$ monic 
of degree $d \ge 2$  and let $P \in \C[x , y]$. Then
	\[
	\limsup_{n\to\infty} \m_f(P(x, f^n(x))) \le \m_f(P(x, y)).
	\]
\end{prop}
The above statement is expected to be true with an equality. This would correspond to a dynamical version of the Boyd--Lawton limit \eqref{eq:BL}. 

For the main result, we will assume the following conjecture. 
\begin{conj}[Dynamical Lehmer's conjecture] \cite[Conjecture~3.25]{Silverman-arithmetic-dynamical} \label{conj:dynamicalLehmer}
 There is some $\delta = \delta_f > 0$ such that any  single variable polynomial $P\in \Z[x]$ with $\m_f(P) > 0$ satisfies $\m_f(P) > \delta$.
\end{conj}

 \begin{thm} \label{thm:mainresult}
 	Assume the Dynamical Lehmer's conjecture.
 	
 	Let $f \in \Z[z]$ be a monic polynomial of degree $d \ge 2$ which is not conjugate to $z^d$ or to $\pm T_d(z)$, where $T_d(z)$ is the $d$-th Chebyshev polynomial.
 	Then any polynomial $P \in \Z[x, y]$ which is irreducible in $\Z[x,y]$  (but not necessarily irreducible in $\C[x, y]$)
 	with $\m_f(P) = 0$, and which contains both variables $x$ and $y$,
 	divides a product of complex polynomials of  the following form:
 	 \[	\tilde{f}^n(x)  - L(\tilde{f}^m(y)),\]
  where $m, n \ge 0$ are integers, $L \in \C[z]$ is a linear polynomial commuting with an iterate of $f$, and $\tilde{f} \in \C[z]$ is a nonlinear polynomial of minimal degree commuting with an iterate of $f$ (with possibly different choices of $L$, $\tilde{f}$, $n$, and $m$ for each factor).

As a partial converse, suppose there exists a product of complex polynomials $F_j$ such that
		\begin{enumerate}
		\item each $F_j$ has the form $\tilde{f}^{n}(x) - L(\tilde{f}^{m}(y))$, where $L$ and $\tilde{f}$ are as above (with possibly different choices of $L$, $\tilde{f}$, $n$, and $m$ for each $j$);
		\item $\prod F_j \in \Z[x, y]$; and
		\item $P$ divides $\prod F_j$ in $\Z[x, y]$.
		\end{enumerate}
	Then $\m_f(P) = 0$.
  
%(Here the colexicographic order of monomials is defined by $x^{a} y^b \ge x^{a'} y^{b'}$ if $b > b'$ or $b=b'$ and $a \ge a'$.  That is, for a set of monomials the one that is largest in the colexicographical ordering is, among the ones that has highest degree in $y$, the one with highest degree in $x$.)
\end{thm}

Note that the second half of Theorem~\ref{thm:mainresult} is only a partial converse.  We can guarantee that $\prod F_j \in \Q[x, y]$ by enlarging the set $\{F_j\}$ to contain all the Galois conjugates of all its members.   However, as we do not know whether the polynomials $F_j$ have algebraic integer coefficients, we cannot guarantee $\prod F_j \in \Z[x, y]$.  We conjecture however that it is always the case that the $F_j$ have algebraic integer coefficients, in which case we could strengthen this converse.

We remark that Theorem~\ref{thm:mainresult} complements a statement for classical Mahler measure. Indeed, it is shown in \cite[Theorem~3.10]{EverestWard} that for any primitive polynomial $P\in \Z[x_1^\pm, \dots, x_n^\pm]$, $\m(P)$ is zero if and only if $P$ is a monomial times a product of cyclotomic polynomials evaluated on monomials. This case is fundamentally different from what is considered in Theorem \ref{thm:mainresult}, as the classical Mahler measure contains more structure given by the multiplicative nature of the function $f(z)=z^d$. 

It is also interesting to compare our results with \cite[Conjecture~2.5]{Zhang}, which says that a subvariety whose dynamical canonical height is $0$ must be preperiodic.  If Zhang's height agrees with the dynamical Mahler measure, as we expect to be the case, then Lemma~\ref{lem:dynamicalKronecker} and Theorem~\ref{thm:mainresult} can be seen as special cases of Zhang's Conjecture 2.5 for the dynamical maps $f: \PP^1 \to \PP^1$ and $f \times f: \PP^1 \times \PP^1 \to \PP^1 \times \PP^1$, respectively.  (Note that Zhang's conjecture is known not to hold in general: see \cite{Ghioca-Tucker-Zhang} for a counterexample.  However, it may be possible to use the arguments in \cite{Ghioca-Nguyen-Ye-two-var} to obtain a proof for the case of $f \times f: \PP^1 \times \PP^1 \to \PP^1 \times \PP^1$.)

This article is organized as follows.  Section \ref{sec:Ransford-is-our-friend} covers some background on complex dynamics. Section \ref{sec:well-defined} is devoted to showing that  the dynamical Mahler measure is well-defined in general, and non-negative for polynomials with integral coefficients, while Section \ref{sec:Kronecker} discusses the dynamical version of Kronecker's Lemma \ref{lem:dynamicalKronecker}. The Weak Dynamical Boyd--Lawton Proposition \ref{prop:weakdynamicalBL} is discussed and proven in Section \ref{sec:weakBL}. It serves as a preparation for the proof of our main result, Theorem \ref{thm:mainresult}, covered in Sections \ref{sec:zeromeasure} and \ref{sec:2-Kronecker}. Finally, in Section \ref{sec:Lf}, we give a precise description of the polynomials $L$ and $\tilde{f}$ that appear in the statement of this theorem.

\subsection*{Acknowledgements} We are grateful to Patrick Ingram for proposing that we study the dynamical Mahler measure of multivariate polynomials and for many early discussions, and to the anonymous referees for their careful reading of the article and several helpful suggestions. This project was initiated as part of the BIRS workshop ``Women in Numbers 5'', held virtually in 2020. We thank the workshop organizers,  Alina Bucur, Wei Ho, and Renate Scheidler for their leadership and encouragement that extended for the whole duration of this project. This work has been partially supported by 
the Natural Sciences and Engineering Research Council
of Canada (Discovery Grant 355412-2013 to ML), the Fonds de recherche du Qu\'ebec - Nature et technologies (Projets de recherche en \'equipe 256442 and 300951 to ML), the Simons Foundation (grant number 359721 to MM), and the National Science Foundation (grant DMS-1844206 supporting AC and grant  DMS-1902772 to LM).
 This material is based upon work supported by and while the third author served at the National Science Foundation. Any
opinion, findings, and conclusions or recommendations expressed in this material are those of the authors
and do not necessarily reflect the views of the National Science Foundation.

Data sharing not applicable to this article as no datasets were generated or analyzed during the current study.

\section{Some background on complex dynamics}\label{sec:Ransford-is-our-friend}

In this section we recall some results from complex dynamics that will be useful for the rest of this article. The main reference for this is the book of Ransford \cite{Ransford}. 

Let $\nu$ be a finite Borel measure on $\C$ with compact support $\K$. One can define the potential function  \cite[Definition 3.1.1]{Ransford}  \[p_\nu: \C \rightarrow[-\infty, \infty)\] by 
\begin{equation}\label{eq:potential}
p_\nu(z):=\int_{\K} \log|z-w|d\nu(w).
\end{equation}
The potential $p_\nu$ is subharmonic in $\C$ by \cite[Theorem 3.1.2]{Ransford}.

Thus we can write the energy as 
\[I(\nu):=\int_\K \int_\K \log|z-w|d \nu(z) d \nu(w)=\int_\K p_\nu(z) d\nu(z).\]
Then Frostman's Theorem 
\cite[Theorem 3.3.4]{Ransford} implies that if $\mu$ is the equilibrium measure of $\mathcal{K}$, then 
\begin{enumerate}
\item $p_\mu \geq I(\mu)$ on $\C$;
\item $p_\mu = I(\mu)$ on $\mathcal{K}\setminus \mathcal{E}$, where $\mathcal{E}\subseteq \partial{\mathcal{K}}$.
\end{enumerate}
By \eqref{eq:energycoeff}, $I(\mu_f)\not = -\infty$ and $p_{\mu_f}$ is a finite number.

%The logarithmic capacity of a compact set $\mathcal{K}\subset \C$ is given by 
%\[c(K)=e^{I(\mu)},\]
%where $\mu$ is the equilibrium measure (\cite[Definition 5.1.1.]{Ransford}). 

%Finally by \cite[Theorem 6.5.1]{Ransford}, if $f(z)=\sum_{j=0}^d a_jz^j$ is a polynomial of degree $d\geq 2$, then its Julia set has capacity $c(\J_f)=\frac{1}{|a_d|^{1/(d-1)}}$. 
%In particular, the capacity is non-zero, and this implies that $I(\mu_f)\not = -\infty$ and $p_{\mu_f}$ is a finite number.  

Let  $\mathcal{D}$ a proper sub-domain of the Riemann sphere. A Green's function \cite[Definition 4.4.1]{Ransford} for $\mathcal{D}$ is a map $g_\mathcal{D}: \mathcal{D}\times \mathcal{D}\rightarrow (-\infty,\infty]$ such that for each $w \in \mathcal{D}$:
\begin{enumerate}
 \item $g_\mathcal{D}(\cdot, w)$ is harmonic on $\mathcal{D}\setminus \{w\}$, and bounded outside each neighbourhood of $w$;
 \item $g_\mathcal{D}(w,w)=\infty$, and as $z \rightarrow w$, 
 \[g_\mathcal{D}(z,w)=\begin{cases}
                       \log|z|+O(1) & w=\infty,\\
                       -\log|z-w|+O(1) & w \not = \infty;
                      \end{cases}\]
\item     $g_\mathcal{D}(z,w)\rightarrow 0$  as $z \rightarrow \zeta$ for $\zeta\in \partial \mathcal{D}$ outside of a Borel polar subset of $\partial \mathcal{D}$.                   
\end{enumerate}

A set $\mathcal{E}\subset \C$  is called polar if $I(\nu)=-\infty$ for every finite Borel measure $\nu \not= 0$ for which $\mathrm{supp}\, \nu$ is a compact subset of $\mathcal{E}$ \cite[Definition 3.2.2]{Ransford}.  Otherwise $\mathcal{E}$ is called non-polar. If $\mathcal{D}$ is a domain of the Riemann sphere such that $\partial \mathcal{D}$ is non-polar, then,  there is a unique Green's function $g_\mathcal{D}$ for $\mathcal{D}$ \cite[Theorem 4.4.2]{Ransford}. 

In particular, the proof of \cite[Theorem 4.4.2]{Ransford} gives us explicitly that, if $\partial \mathcal{D}$ is non-polar and if $\mu$ is the equilibrium measure on $\mathcal{D}$, then
\begin{equation}\label{green and potential}
g_\mathcal{D}(z, \infty) =  p_{\mu}(z) - I(\mu)
\end{equation}
for $z \in \mathcal{D} \setminus \{\infty\}$.

If $\K_f$ is the filled Julia set, we remark that \cite[Corollary 6.5.4]{Ransford} implies that the Green's function of its complement is given by  
\[g_{\mathbb{P}^1(\C)\setminus \K_f}(z,\infty)=\lim_{n\rightarrow \infty} \frac{1}{\deg(f)^n} \log^+|f^n(z)|,\]
where $\log^+|\alpha|=\log|\alpha|$ when $|\alpha|>1$ and 0 otherwise. 

 From now on, we will write $g_f(z,w)$ for the Green's function of the complement of the filled Julia set.
 
We can summarize the results of the discussion above in the following two propositions:

\begin{prop}
	If $f \in \C[z]$ monic of degree $d\geq 2$, then 
	\begin{equation}\label{eq:g=p}
		g_{f}(z, \infty) = p_{\mu_f}(z)
	\end{equation}
for $z \notin \K_f$.
\end{prop}
\begin{proof}
	By \cite[Theorem~6.5.1]{Ransford}, the Julia set $\J_f = \partial \K_f$ is non-polar, and so equation \eqref{green and potential} applies. Since $f$ is monic, by \eqref{eq:energycoeff},  we have $I(\mu_f) = 0$.
\end{proof}
 
\begin{prop}\label{prop:nonnegpotential}
	If $f \in \C[z]$ is monic of degree $d\geq 2$, then $ p_{\mu_f}(z) \ge 0$ for all $z \in \C$, and the equality $ p_{\mu_f}(z) = 0$ holds exactly when $z \in \K_f$.
\end{prop}
\begin{proof}
If $f$ is monic, we have $I(\mu_f) = 0$ by \eqref{eq:energycoeff} (see also \cite[Theorem 6.5.1]{Ransford}),
 and since $p_{\mu_f}\geq I(\mu_f)$, we conclude that 
$ p_{\mu_f}\geq 0.$ 
By Frostman's Theorem~\cite[Theorem 3.3.4]{Ransford}, we have $p_{\mu_f}(z)=I(\mu_f)$ for $z$ in the interior of $\K_f$, and since $\PP^1(\C) \setminus \K_f$ is a regular domain~\cite[Corollary~6.5.5]{Ransford}, we have $p_{\mu_f}(z) = I(\mu_f)$ on $\J_f$ also~\cite[Theorem~4.2.4]{Ransford}.

Conversely, since the Julia set $\J_f = \partial \K_f$ is  non-polar, we have $g_{f}(z,\infty)>0$ for $z\not \in \K_f$~\cite[Theorem 4.4.3]{Ransford}. 
By equation \eqref{eq:g=p}, we conclude that $p_{\mu_f}(z)=0$ if and only if $z \in \K_f$. 
\end{proof}

 \section{Convergence and positivity of the dynamical Mahler measure} \label{sec:well-defined}
 We will prove that the dynamical Mahler measure defined by \eqref{eq:dmm-defi} is well-defined. Moreover, it is  non-negative when $P \in \Z[x_1,\dots,x_n]$ is nonzero. We will follow a structure of a proof that is similar to that of \cite[Lemma 3.7]{EverestWard}. 
 
 We start first with the following result, which  can be considered an analogue to Jensen's formula in the dynamical universe. 
\begin{lem}\label{lem:dynamicaljensen}
  Let $f \in \Z[z]$ monic  of degree $d\geq 2$. 	If $P(x) = a \prod_{i} (x-\alpha_i)\in \C[x]$, then 
\[
\m_f(P) = \log|a| + \sum_{i} p_{\mu_f}(\alpha_i).
\]
\end{lem}
\begin{proof} By definition, we have
	\[	
	\m_f(P)  = \int_{\J_f} \log \big|a \prod_{i} (z-\alpha_i)\big| d\mu_f (z)= \int_{\J_f}  \log |a| d\mu_f(z)+ \sum_{i}\int_{\J_f} \log|z-\alpha_i|d\mu_f(z),
	\]
	and the result follows from equation~\eqref{eq:potential}.
\end{proof}
 
 Now we proceed to prove the main result of this section. 
  \begin{prop}
  Let $f \in \Z[z]$ monic  of degree $d\geq 2$ and let $P\in  \C(x_1,\dots,x_n)^\times$. Then the integral in \eqref{eq:dmm-defi} defining the $f$-dynamical Mahler measure converges. Moreover, if $P \in \Z[x_1,\dots,x_n]$ is nonzero, then $\m_f(P)\geq 0$. 
 \end{prop}
 
 \begin{proof}
First assume that  $P\in  \C[x_1,\dots,x_n]$ is nonzero. Let $C=\sup_{z \in \J_f} |z|$, and write 
  \[P(x_1,\dots,x_n)=\sum a_{i_1,\dots,i_n}x_1^{i_1}\cdots x_n^{i_n}.\]
   Then for $(z_1,\dots,z_n)\in \J_f^n$, we have 
   \[|P(z_1,\dots,z_n)|\leq \sum |a_{i_1,\dots,i_n}|C^{i_1+\cdots+i_n}=:L_C(P).\]
  From this, we see that $\m_f(P)\leq \log( L_C(P))<\infty$. 

  For the rest of the statement, we proceed by induction. The case $n=1$ is a consequence of what we have just discussed and Lemma \ref{lem:dynamicaljensen}, which implies that $\m_f(P)>-\infty$ when $P\in \C[x]$ is nonzero and $\m_f(P)\geq 0$ when $P\in \Z[x]$ is nonzero. Assume that the result is proven for all polynomials in $n-1$ variables. We write 
  \begin{align*}
  P(x_1,\dots,x_n) &= a_d(x_2,\dots, x_n)x_1^d+\cdots +a_0(x_2,\dots,x_n)\\
  &= a_d(x_2,\dots, x_n) \prod_{j=1}^d (x_1-g_j(x_2,\dots,x_n)),
\end{align*}
for certain algebraic functions $g_1,\dots,g_d$. Then we have that 
\begin{align*}
 \int_{\J_f} \dots \int_{\J_f} &\log | P(z_1, \dots, z_n)| d\mu_f(z_1) \cdots d\mu_f(z_n)\\
 &= \m_f(a_d)+ \sum_{j=1}^d
 \int_{\J_f} \dots \int_{\J_f} \log |z_1-g_j(z_2,\dots,z_n)| d\mu_f(z_1) \cdots d\mu_f(z_n) \\
 &= \m_f(a_d)+\sum_{j=1}^d
 \int_{\J_f} \dots \int_{\J_f}p_{\mu_f}(g_j(z_2,\dots,z_n))d \mu_f(z_2) \cdots d\mu_f(z_n) \\
 &= \m_f(a_d)+
 \int_{\J_f} \dots \int_{\J_f} \sum_{j=1}^d p_{\mu_f}(g_j(z_2,\dots,z_n))d \mu_f(z_2) \cdots d\mu_f(z_n).
\end{align*}
Here our final integrand is upper semicontinuous, since $p_{\mu_f}$ is upper semicontinuous and the multiset of values $\{g_i(z_2, \dotsc, z_n)\}$ depends continuously on $(z_2, \dotsc, z_n)$ (even though the individual $g_i$ are not themselves continuous at the branch locus).  By Proposition \ref{prop:nonnegpotential}, $p_{\mu_f}\geq 0$.
By induction,  $\m_f(a_d)\geq 0$. Therefore the above integral is non-negative and the first paragraph shows that the entire right-hand side is bounded above. Therefore the integral defining the Mahler measure exists.

Now assume that $P\in \C(x_1,\dots, x_n)^\times$. We write $P=\frac{F}{G}$, with $F, G \in  \C[x_1,\dots, x_n]$ nonzero and apply the above to conclude that $\m_f(F), \m_f(G)$ are well-defined. Then we have $\m_f(P)=\m_f(F)-\m_f(G)$ and $\m_f(P)$ is well-defined too. This concludes the proof of the statement. 
 \end{proof}

\section{Dynamical Kronecker's Lemma} \label{sec:Kronecker}

The goal of this section is to characterize the single variable polynomials with integral coefficients $P$ having $\m_f(P)=0$, which is the dynamical version of Kronecker's Lemma:
\begin{lem}[Kronecker, \cite{Kronecker}] Let $P(x)=\prod_i (x-\alpha_i) \in \Z[x]$. If all of the roots of $P$ satisfy $|\alpha_i | \leq 1$,  then the $\alpha_i$ are zero or roots of unity.
\end{lem}
In fact, more is true. Equation \eqref{eq:onevariable} implies the following statement. 
\begin{cor}
 Let $P \in \Z[x]$. If  $\m(P)=0$, then the roots of $P$ are either zero or roots of unity. Conversely, if $P$ is primitive and its roots  either zero or roots of unity, then $\m(P)=0$. 
\end{cor}

An immediate consequence of Lemma \ref{lem:dynamicaljensen} is the following. 
\begin{lem} \label{lem:cortodynamicalJensen}  Let $f\in \Z[z]$ be monic of degree $d\geq 2$ and $P(x)=a\prod_i (x-\alpha_i) \in \Z[x]$. Then we have $\m_f(P)=0$ if and only if $|a| = 1$ and all roots $\alpha_i$ lie in  $\K_f$. 
\end{lem}

\begin{proof}
   Because $a \in \mathbb{Z}$, we have $\log |a| \ge 0$, and by Proposition \ref{prop:nonnegpotential},  all the other summands in  the dynamical Jensen formula from Lemma \ref{lem:dynamicaljensen} are also nonnegative. Therefore, $\m_f(P)=0$ if and only if $\log |a|  = 0$ and $p_{\mu_f}(\alpha_i) = 0$ for all $i$.  The first condition $ \log |a| = 0$ holds exactly when $a = \pm 1$. The second condition is an immediate consequence of the second statement in Proposition \ref{prop:nonnegpotential}.
   \end{proof}

We are now ready to prove the Dynamical Kronecker's Lemma (Lemma $\ref{lem:dynamicalKronecker}$).

\begin{proof}[Proof of Lemma $\ref{lem:dynamicalKronecker}$] 
Since Lemma \ref{lem:cortodynamicalJensen}  implies that $P(x)$ is (up to a sign) monic, we can write
\[P(x) =\pm \prod_i (x-\alpha_i).\]
Consider the polynomials 
\[P_n(x)=\prod_{i=1}^d(x-f^n(\alpha_i)).\]
  The coefficients of $P_n$ are symmetric functions in the algebraic integers $f^n(\alpha_i)$ and thus symmetric functions of the $\alpha_i$, so they are elements of $\Z$ (all the conjugates of each $\alpha_i$ are present as roots of $P$, since the coefficients are rational). 
Since $\m_f(P)=0$, Lemma \ref{lem:cortodynamicalJensen}  implies that $\alpha_i\in \K_f$, and the same is true for $f^n(\alpha_i)$. Since $\K_f$ is compact,  
  the $f^n(\alpha_i)$ are uniformly bounded, and the same is true for the coefficients of $P_n$. 
  Thus, the set $\{P_n\}_{n\in\N}$ must be finite. In other words, there are 
  $n_1\not = n_2$ for which
  \[P_{n_1}=P_{n_2}.\]
That means,
\[\{f^{n_1}(\alpha_1),\dots,f^{n_1}(\alpha_d)\}=\{f^{n_2}(\alpha_1),\dots,f^{n_2}(\alpha_d)\}.\]
Thus, there is a permutation $\sigma \in \mathbb{S}_d$ such that
\[f^{n_1}(\alpha_i)=f^{n_2}(\alpha_{\sigma(i)}).\]
If $\sigma$ has order $k$, we get,
\[f^{kn_1}(\alpha_i)=f^{kn_2}(\alpha_{i}),\]
which shows that each $\alpha_i$ is preperiodic. 
 \end{proof}

\begin{remk}
A more general approach to dynamical Mahler measure, for example as suggested in~\cite{{Favre-Rivera-Letelier}}, may allow $\m_f$  to be defined for $f$ a non-monic polynomial or indeed a rational function. The following argument shows that Lemma~\ref{lem:dynamicalKronecker} continues to hold in this case:
 Noting that $\log|PQ| = \log|P| + \log|Q|$, we see
from equation~\eqref{eq:adelic} that $\m_f(P)$ vanishes  with the $f$-canonical height of the roots of $P$. 
Over a global field, the canonical height $\hat{h}_f$ of a rational function vanishes precisely at preperiodic points for $f$  (see, for example, \cite[Theorem 3.22]{Silverman-arithmetic-dynamical}). 
\end{remk}

\section{Weak Dynamical Boyd--Lawton Theorem } \label{sec:weakBL}

A result of  Boyd \cite{Boyd-speculations} and Lawton~\cite{Lawton} given by equation \eqref{eq:BL}
allows one to compute, in the classical setting, the higher dimensional Mahler measure by a series of approximations by one-dimensional Mahler measures. Here we exhibit the dynamical version of a weaker result for two variables, Proposition \ref{prop:weakdynamicalBL}. Our  result is analogous to Lemma 2 in \cite{Boyd-Kronecker}.

First we need to establish the following lemma, which is analogous to Lemma 1 in \cite{Boyd-Kronecker}.
\begin{lem}\label{lem:Weierstrass}
If $F : \J_f^2 \to \R$ is continuous, then
\[
\lim_{n\to \infty} \int_{\J_f} F(z, f^n(z)) d\mu_f(z) = \int_{\J_f} \int_{\J_f} F(z_1, z_2) d\mu_f(z_1)d\mu_f(z_2) .
\]
\end{lem}

\begin{proof}
By making the change of variables $z_2 = f^n(z)$, we have 
\begin{equation}\label{change var}
	\int_{\J_f} F(z, f^n(z)) d \mu_f(z) = \int_{z_2 \in \J_f} d^{-n} \sum_{z_1 \in f^{-n}(z_2)} F(z_1, z_2) d\mu_f(z_2),
\end{equation}
where $d=\deg(f)$. 
We use this to expand our left hand side,  interchanging the limit and integral by the bounded convergence theorem.

\begin{equation*}
	\begin{split}
		\lim_{n \to \infty} \int_{\J_f} F(z, f^n(z)) d \mu_f(z) 
		& = \lim_{n \to \infty} \int_{z_2 \in \J_f} d^{-n} \sum_{z_1 \in f^{-n}(z_2)} F(z_1, z_2) d\mu_f(z_2)\\ & = \int_{z_2 \in \J_f} \left (\lim_{n \to \infty} d^{-n} \sum_{z_1 \in f^{-n}(z_2)} F(z_1, z_2) \right) d\mu_f(z_2)
	\end{split}
\end{equation*}
By \cite[Theorem 6.5.8]{Ransford}, we have that the limit inside the parentheses is equal to $\int_{z_1 \in \J_f} F(z_1, z_2) d\mu_f(z_1)$, as desired.
\end{proof}

We are now ready to prove the Weak Dynamical Boyd--Lawton Theorem. %(Proposition \ref{prop:weakdynamicalBL}).

\begin{proof}[Proof of Proposition  $\ref{prop:weakdynamicalBL}$] 
	We follow Lemma 2 in \cite{Boyd-Kronecker}. First notice that for $\delta>0$, $\log\left(|P(z_1,z_2)|+\delta\right)$ is a continuous function on $\J_f$. Therefore, by Lemma \ref{lem:Weierstrass}, 
\[\lim_{n\to\infty} \int_{\J_f}\log (|P(z, f^n(z))|+\delta) d\mu_f(z)=\int_{\J_f}\int_{\J_f}\log (|P(z_1,z_2)|+\delta) d\mu_f(z_1)d\mu_f(z_2).\]
	Since $\m_f(P(z, f^n(z)) \leq \int_{\J_f}\log (|P(z, f^n(z))|+\delta) d\mu_f(z)$, we have
	\[\limsup_{n \rightarrow \infty} \m_f(P(z, f^n(z))) \leq  \int_{\J_f}\int_{\J_f}\log (|P(z_1,z_2)|+\delta) d\mu_f(z_1)d\mu_f(z_2),\]
	and this is true for all $\delta>0$.  Letting $\delta \rightarrow 0$ and using monotone convergence gives the result. 
\end{proof}

\section{A family of polynomials having dynamical Mahler measure zero} \label{sec:zeromeasure}

In Section \ref{sec:Kronecker} we discussed necessary and sufficient conditions for a single variable polynomial with integer coefficients to have dynamical Mahler measure zero.  In this section we will show that the dynamical Mahler measure of $P(x, y) = x - y$ is zero independently of $f$ as long as $f$ is monic, and we will also use this fact to construct 
 a more general class of polynomials whose dynamical Mahler measure is zero. This is the beginning of the proof of Theorem \ref{thm:mainresult}, which will be finished  in Section \ref{sec:2-Kronecker}.

Recall that the multivariate dynamical Mahler measure 
\[
\m_f(P) = \int _{\J_f}\int _{\J_f} \log |z_1-z_2| d\mu_f(z_1) d \mu_f(z_2)
\]
is precisely the energy $I(\mu_f)$ of the equilibrium measure. By \eqref{eq:energycoeff}, when $f$ is monic, $I(\mu_f)=0$, and therefore $\m_f(P) =0$. 

%By  \cite[Theorem 6.5.1]{Ransford}, when $f$ is monic, the Julia set $\J_f$ has capacity $1$, and so its equilibrium measure $\mu_f$ has energy $0$.

Notice that the strong Boyd--Lawton Conjecture, which claims that $\m_f(P(x,y))=\lim_{n\rightarrow \infty} \m_f(P(x,f^n(x))$,  applies in this case. Indeed, for $n > 0$, the polynomial $P_{n}(x) = x-  f^n(x) $ has roots equal to the $n$-periodic points of $f$, so 
\[
\m_f(P_n) = \sum_{f^n(z) = z} g_f(z, \infty),
\]
where $g_f(z,\infty)$ denotes the Green's function of the filled Julia set $\K_f$.
However if $z$ is periodic, then $g(z,\infty) = \lim_{m \to \infty} d^{-m} \log^+|f^m(z)|$ is $0$ since the log term is bounded.
Hence $\lim_{n \to \infty} \m_f(P_n) = 0$, and this coincides with $\m_f(P)=0$, which is what one would predict from the  strong Boyd--Lawton Conjecture.

Now we intend to prove a more general result. 
\begin{lem}\label{lem:converseKronecker}Let $L$ be a linear polynomial in $\C[z]$ commuting with an iterate of $f$, and $\tilde{f}$ be a nonlinear polynomial in $\C[z]$ commuting with an iterate of $f$. Then for any nonnegative integers $n, m$,
 \[\m_f\Big(\tilde{f}^n(x)-L(\tilde{f}^m(y))\Big)=0.\]
\end{lem}
To prove Lemma~\ref{lem:converseKronecker}, we first need a lemma on dynamical systems.

\begin{lem}\label{lem:invariance}
	Let $g \in \C[z]$ be any polynomial commuting with a power of  $f$ (including the case of $g$ linear).  Then $g$ sends the Julia set $\J_f$ to itself. Furthermore, $g$ preserves the invariant measure $\mu_f$.   Explicitly, this means that for any test function $\phi$, $\int_{\J_f} \phi(z) d \mu_f = \int_{\J_f} \phi(g(z)) d \mu_f$. 
\end{lem}

\begin{proof}
   Since $f$ and $f^n$ have the same Julia set and invariant measure, we can replace $f$ with $f^n$ to reduce to the case where $g$ commutes with $f$. 
   
   Since $g$ commutes with $f$, the measure $g_* \mu_f$ on $\C$ defined by $g_* \mu_f (X) = \mu_f(g^{-1}(X))$ is also invariant under $f$ and has total measure $1$.  By uniqueness of the invariant measure, we have $g_* \mu_f = \mu_f$ as desired.  Since $\J_f$ is the support of $\mu_f$, it too is preserved by $g$.  
\end{proof}

\begin{proof}[Proof of Lemma~$\ref{lem:converseKronecker}$] 

We apply the invariance of measure in Lemma~\ref{lem:invariance} to both variables in this integral, and obtain
\[
\int _{\J_f} \int _{\J_f} \log |\tilde{f}^n(z_1)-L(\tilde{f}^m(z_2))| d\mu_f(z_1) d \mu_f(z_2)
= \int _{\J_f} \int _{\J_f} \log |z_1-z_2| d\mu_f(z_1) d \mu_f(z_2),
\]
which we know to be $0$.
\end{proof}

\begin{remk} We will see in Section \ref{sec:Lf}, Proposition \ref{prop:ftildeL} that for $f \in \Z[z]$ both $L$ and $\tilde{f}$ have algebraic coefficients. 
\end{remk}

 As a corollary, we obtain a result proving part of Theorem~\ref{thm:mainresult}.
 \begin{cor} \label{coro:zeromeasure}
 	Let $f \in \Z[z]$ be monic of degree $d\geq 2$, and let $P \in \Z[x, y]$ be nonzero. Suppose there exists a product of complex polynomials $F_j$ such that
		\begin{enumerate}
		\item each $F_j$ has the form $\tilde{f}^{n}(x) - L(\tilde{f}^{m}(y))$, where $L$ is a linear polynomial commuting with an iterate of $f$ and $\tilde{f}$ is a nonlinear polynomial of minimal degree commuting with an iterate of $f$ (with possibly different choices of $L$, $\tilde{f}$, $n$, and $m$ for each $j$);
		\item $\prod F_j \in \Z[x, y]$; and
		\item $P$ divides $\prod F_j$ in $\Z[x, y]$.
		\end{enumerate}
	Then $\m_f(P) = 0$.
 \end{cor}
 
 \begin{proof}
 	We have \[
 \prod_{j}F_j(x,y) = P(x, y) Q(x, y)\] for some $Q \in \Z[x, y]$.  Taking Mahler measures of both sides, we get:
 \[
 0 = \m_f(P) + \m_f(Q)
 \]
 and both summands are nonnegative (since $P$ and $Q$ have integer coefficients), so we must have that $\m_f(P) = \m_f(Q) = 0$.
 \end{proof}

 \section{Two-variable Dynamical Kronecker} \label{sec:2-Kronecker}
 
 Our goal in this section is  to finish the proof of our main result, Theorem \ref{thm:mainresult}. More precisely, in Corollary \ref{coro:zeromeasure} we proved that certain polynomials have dynamical Mahler measure zero. It is now time to prove that those are the only polynomials that satisfy this condition. In other words we will prove the following result. 
 
 \begin{thm} \label{thm:babymainresult}
 	Assume the Dynamical Lehmer's Conjecture $\ref{conj:dynamicalLehmer}$. 
 	
 	Let $f \in \Z[z]$ be a  monic polynomial of degree $d \ge 2$ which is not conjugate to $z^d$ or to $\pm T_d(z)$, where $T_d(z)$ is the $d$-th Chebyshev polynomial.
 	Let $P \in \Z[x, y]$ be such that  $\m_f(P) = 0$, $P$ is irreducible in $\Z[x, y]$, and $P$ uses both variables. 
    Then every irreducible factor of $P$ in $\C[x, y]$
 divides a polynomial of the following form:
 	\[	\tilde{f}^n(x)  - L(\tilde{f}^m(y)),\]
  where $m, n \ge 0$ are integers, $L$ is a linear polynomial commuting with an iterate of $f$, and $\tilde{f}$ is a nonlinear polynomial of minimal degree commuting with an iterate of $f$.

 \end{thm}	 

In order to prove Theorem \ref{thm:babymainresult} we will need to assume the Dynamical Lehmer's Conjecture \ref{conj:dynamicalLehmer}. In addition, we will use  Proposition \ref{prop:weakdynamicalBL}.
%, but in fact, we will only need \[\liminf_{n\to\infty} \m_f(P(x, f^n(x))) \le \m_f(P(x, y))\] for our application. 
Finally, we will also need the following unlikely intersections result due to Ghioca, Nguyen and Ye \cite{Ghioca-Nguyen-Ye-two-var}.

%(Ghioca, D.; Nguyen, K. D.; Ye, H. The dynamical Manin-Mumford conjecture and the dynamical Bogomolov conjecture for split rational maps. J. Eur. Math. Soc. (JEMS) 21 (2019), no. 5, 1571--1594. MR3941498) 

\begin{prop}\cite[Theorem 1.5]{Ghioca-Nguyen-Ye-two-var} \label{prop:GNY}
	Let $f \in \C[z]$ be a polynomial of degree $d \ge 2$ which is not conjugate to $z^d$ or to $\pm T_d(z)$, and let $\Phi: \mathbb{P}^1 \times \mathbb{P}^1 \to \mathbb{P}^1 \times \mathbb{P}^1$ be defined by $\Phi(x, y) = (f(x), f(y)).$  
Let $C \subset \mathbb{P}^1 \times \mathbb{P}^1$ be an irreducible curve defined over $\C$ which projects dominantly onto both coordinates.  Then $C$ contains infinitely many preperiodic points under the action of $\Phi$ if and only if $C$ is an irreducible component of the locus of an equation of the form $\tilde{f}^n(x)  = L(\tilde{f}^m(y))$, where $m, n \ge 0$ are integers, $L$ is a linear polynomial commuting with an iterate of $f$, and $\tilde{f}$ is a nonlinear polynomial of minimal degree commuting with an iterate of $f$.
\end{prop}

The following definition will be key to our proof:
\begin{defn}
\label{defn:BOP}	The triple $(f, P, \alpha)$ with $f \in \C[z]$,
$P \in \C[x, y]$, and $\alpha \in \C$  satisfies the bounded orders property (BOP) if the set of nonnegative integers
	\[
	\{\ord_{\alpha} P(x, f^n(x))\}
	\]
	is bounded. (In the above sequence we ignore any terms where $P(x,f^n(x))$ is the zero polynomial, which can happen for only finitely many $n$.)
	
	We say that $(f, P)$ satisfies the preperiodic bounded orders property (preperiodic BOP) if $(f, P, \alpha)$ has BOP for every $\alpha \in \C$ such that $\alpha$ is a preperiodic point of $f$.
	 \end{defn}

The following general reduction properties follow directly from the definition. We omit their proof.  
\begin{lem}\label{lem:easy} Let $f \in \C[z]$ and $P \in \C[x, y]$.
\begin{enumerate}
	\item (Translation) Let $P_{-\alpha}(x,y)=P(x+\alpha,y+\alpha)$ and $f_{-\alpha}(z) = f(z+\alpha)-\alpha$.
	If $(f_{-\alpha}, P_{-\alpha}, 0)$ has BOP, then $(f, P, \alpha)$ has BOP. 
	\item (Iteration) For $0 \le i < k$ let $P \circ (\id \times f^i)$ be given by $P \circ (\id \times f^i)(x, y) = P(x, f^i(y))$.  If $(f^k, P \circ (\id \times f^i), \alpha)$ has BOP for $0 \le i < k$, then $(f,P,\alpha)$ has BOP.
	\item (Divisibility) If $(f, P, \alpha)$ has BOP and $Q \mid P$, then $(f, Q, \alpha)$ has BOP.
\end{enumerate}	
\end{lem}

Now we consider some less immediate reduction properties.
\begin{lem}[Composition/Pushforward reduction] \label{lem:trickier}	Suppose that $f \in \C[z]$ and $P \in \C[x, y]$. 
\begin{enumerate}
\setcounter{enumi}{3}
\item     (Composition) Let $f^* P = P \circ  (f \times f)$ be the polynomial given by $f^* P = P(f(x), f(y))$. 
    If $(f, P, f(\alpha))$ has BOP, then $(f, f^*P, \alpha)$ has BOP.
	
\item 	(Pushforward) For  $P(x,y)$ irreducible in $\C[x, y]$, we define $f_*P(x, y)$ to be an irreducible polynomial in $\C[x, y]$ which vanishes on the curve $\{(f(x), f(y)) \, :\, P(x, y) = 0\}$ ($f_*P(x, y)$ is defined up to multiplication by a non-zero scalar\footnote{The fact that the pushforward is only defined up to multiplication by a scalar is not significant, since we only use this construction in the context of order of vanishing.}).  More generally, we can extend this definition multiplicatively: for any $P(x, y) \in \C[x, y]$ with irreducible factorization $P = \prod_{j} F_j$, define $f_* P = \prod_{j} f_* F_j$.

	Then for any $P$ in $\C[x, y]$, if $(f, f_*P, f(\alpha))$ has BOP, $(f, P, \alpha)$, also has BOP.
\end{enumerate}
	
	\end{lem}
\begin{proof}
	(4) Composition: For $Q = f^* P$ and $Q_n(z) = Q(z, f^n(z))$, 
		\begin{align*}
	\ord_{z = \alpha} Q_n(z) &= \ord_{z = \alpha} P(f(z), f^{n+1}(z)) \\
	&= \ord_{z = \alpha} P_n(f(z)) \\
	&= \ord_{z = \alpha} (f(z) - f(\alpha)) \cdot \ord_{z' = f(\alpha)} P_n(z').
	\end{align*}
	Hence the sequence $\ord_{z = \alpha} Q_n(z)$ is bounded if $\ord_{z' = f(\alpha)} P_n(z')$ is.
	
	(5) Pushforward: Assume $(f, f_*P,f (\alpha))$ has BOP.  Since BOP is preserved by taking products, we can without loss of generality assume that $P$ is irreducible.  By the composition property, $(f, f^*f_*P, \alpha)$  also has BOP. However, by construction, the two-variable polynomial $f^*f_*P$  vanishes at any $(x, y)$ with $P(x, y) = 0$, and since $P$ is irreducible, 
	we must have $P \mid f^*f_*P
$ by the Nullstellensatz.  Hence $(f, P, \alpha)$ has BOP by divisibility.
\end{proof}

The motivation for the definition of BOP is the following result. 
\begin{prop} \label{prop:babymainresultBOP}
	Theorem~$\ref{thm:babymainresult}$ holds under the extra hypothesis that $(f, P)$ satisfies preperiodic BOP.
	\end{prop}

\begin{proof}
	Suppose that $P \in \Z[x, y]$ is irreducible and $\m_f(P) = 0$. We consider the family of polynomials $P_n(x) = P(x, f^n(x))$.  
	
	By Weak Dynamical Boyd--Lawton (Proposition \ref{prop:weakdynamicalBL}), we have that $\limsup_{n \to \infty} \m_f(P_n(x)) = 0$. We now apply the Dynamical Lehmer's Conjecture \ref{conj:dynamicalLehmer} to obtain that $\m_f(P_n(x)) = 0$ for $n \gg 0$. Then if $\alpha$ is any root of $P_n(x)$, then, by the single variable Kronecker's Lemma \ref{lem:dynamicalKronecker}, $\alpha$ is preperiodic for $f$.
	
	We now argue that the set $S(P) = \{\alpha \, :\, P_n(\alpha) = 0 \text{ for some }n\}$ is infinite.  Indeed,
	\[
	\deg (P_n) = \sum_{\alpha \in S(P)} \ord_\alpha(P_n(z))
	\]
	goes to $\infty$ as $n \to \infty$, but preperiodic BOP for $(f, P)$ says that the individual summands are bounded, so the sum must be infinite.
	Notice that $S(P)$ is the union of sets $S(Q)$, as $Q$ ranges through the irreducible factors of $P$ in $\C[x,y]$,  and at least one of these sets must be infinite.  However, because $P$ is irreducible in $\Z[x, y]$, the irreducible factors of $P$ form a set of Galois conjugates, and so the sets $S(Q)$ all have the same cardinality, and are all infinite. Hence any irreducible factor $Q$ of $P$ has the property that the set $S(Q) = \{\alpha \, :\, Q_n(\alpha) = 0 \text{ for some }n\}$ is infinite.

	We are now in a position to apply  Proposition \ref{prop:GNY} to $Q$.  Indeed, if $\alpha$ is a root of some $Q_n$, then the curve $Q(x, y) = 0$ passes through the point $(\alpha, f^n(\alpha))$, whose coordinates are both preperiodic for $f$.  Also, the dominance condition in Proposition \ref{prop:GNY} follows from the fact that $P$ uses both variables $x$ and $y$.
	
	Thus, we conclude that each complex factor $Q(x, y)$ of $P(x,y)$ divides a polynomial of the form $\tilde{f}^n(x)  - L(\tilde{f}^m(y))$. 
\end{proof}

It now remains to show the following result. 
\begin{prop}\label{prop:preperiodicBOP} If $f \in \C[z]$ of degree $d\geq 2$
 and $P \in \C[x, y]$ is such that every irreducible factor involves both variables, then $(f, P)$ satisfies preperiodic BOP.
\end{prop}

\begin{proof}
	We must show that if $f$ and $P$ satisfy the above conditions, then for any preperiodic point $\alpha$ of $f$, the triple $(f, P, \alpha)$ satisfies BOP. Our strategy will be to show this in the case where $\alpha$ is a fixed point of $f$, and so we first show that the general case can be reduced to this case.

Assume that we have proved that all such triples $(g, Q, \beta)$ with $\beta$ a fixed point of $g$ have BOP. Take a triple $(f, P, \alpha)$ with $\alpha$ preperiodic for $f$.
Then there exists an $N$ such that $f^{2N}(\alpha) = f^N(\alpha)$, so that $f^N(\alpha)$ is a fixed point for $f^N$. 
By hypothesis, $(f^N, f^N_\ast (P \circ (\id \times f^i)), f^N(\alpha))$ has BOP for each $i$ (note that $f^N_\ast (P \circ (\id \times f^i))$ does not vanish on any horizontal or vertical line because $P$ does not, so all its irreducible factors involve both variables). By pushforward, $(f^N, P\circ(\id \times f^i), \alpha)$ has BOP for each $i$, and by iteration $(f, P, \alpha)$ does also.

We have thus reduced the argument to showing that $(f, P, \alpha)$ has BOP whenever $\alpha$ is a fixed point for $f$. By translation, we can further assume that $\alpha = 0$.  Additionally, if $f'(0)$ is a root of unity of order $k > 1$, then for $g = f^k$, we have that $g'(0) = 1$.  Using the iteration reduction to replace $f$ by $g= f^k$ as necessary, we can assume that if $f'(0)$ is a root of unity, then $f'(0) = 1$.

So suppose that $f \in \C[z]$ is a monic polynomial such that $f(0) = 0$ and $f'(0)$ is either $1$ or not a root of unity, and $P \in \C[x,y]$ is such that every irreducible factor involves both variables.  We show that the sequence $\{ \ord_0 P(z, f^n(z))\}$ is bounded. 

We can factor $P(x, y) = h(x) \prod_{j} (y-g_j(x))$, where $h$ is a nonzero polynomial in $x$ and $g_j$ are non-constant Puiseux series in $x$ (formal power series with possibly fractional exponents, and also possibly a finite number of negative exponents, which are obtained by dividing out by $h$).
Then 
\[
\ord_{0}P(z, f^n(z)) = \ord_0(h(z)) + \sum_{j} \ord_0 (f^n(z) - g_j(z)).
\]
Since $\ord_0(h(z))$ is independent of $n$, it suffices to show that each summand $\ord_0 (f^n(z) - g_j(z))$ is bounded.  Note that if $g_j(z)$ has a term with a negative or non-integer exponent, since $f^n(z)$ is a polynomial, this  bounds $\ord_0(f^n(z) - g_j(z))$, so we can restrict to the case where $g_j \in \C[[z]]$. 

We first deal with the case of $f'(0) = 0$ separately. In this case, $\ord_{0} (f^n(z)) \to \infty$ as $n \to \infty$, so $\ord_{0}(f^n(z) - g(z)) = \ord_{0}(g(z))$ for sufficiently large $n$ and the sequence is bounded.

%Next we note that the case where $f'(0)$ is a $k$-th root of unity can be reduced to the case $f'(0) =1$, by replacing $f$ by $f^k$.

Hence we are reduced to the case $f'(0) \ne 0$, and either $f'(0)$ is not a root of unity or $f'(0) = 1$.  We will use the following notation: if $F \in \C[[z]]$ is a power series, then $[z^i]F(z)$ denotes the coefficient of $z^i$ in $F(z)$.  We have that  $\ord_{0} (f^n(z) - g(z))$ is the smallest $i$ such that $[z^i]f^n(z) \not= [z^i] g(z)$.  The claim states that the sequence of such $i$ is bounded as $n$ varies, which will follow from the following sublemma:

\begin{lem}\label{sublemma}
Let $f(z)$ be a polynomial of degree $d>1$  with $f(0) = 0$ and $f'(0) \neq 0$, and either $f'(0)$ is not a root of unity,  or $f'(0) = 1$. 
There exists some $i$ such that the set of coefficients $\{[z^i]f^n(z)\}_{n \ge 1}$ contains any number at most once. 
\end{lem}
\begin{proof}[Proof of Lemma~\ref{sublemma}]
Since $f(0) = 0$, and $f'(0) \ne 0$, we can write $f(z) = c_1 z + \text{higher order terms}$.  If $c_1$ is not a root of unity, then \[f^n(z) = c_1^n z + \text{higher order terms}\] and $i = 1$ works.

On the other hand, if $c_1 = 1$, we can write $f(z) = z + c_{i_0} z^{i_0} + \text{higher order terms}$, where $c_{i_0}$ is the first nonzero coefficient after the linear one. 
We can calculate
\[
f^n(z) = z + n c_{i_0} z^{i_0} + \text{higher order terms},
\]
so the value $i = i_0$ works.
\end{proof}

Now, for the value of $i$ given by the above lemma, there can be at most one $n$ such that $[z^i]f^n(z) = [z^i]g(z)$, 
and for all other $n$ we have  $\ord_{0} (f^n(z) - g(z)) \le i$.  Hence the values of $\ord_{0} (f^n(z) - g(z))$ 
are bounded.  This finishes the proof in the case that $\alpha$ is a fixed point of $f$, and so by our reductions, also the proof of Proposition \ref{prop:preperiodicBOP}.
\end{proof}

Combining Propositions~\ref{prop:babymainresultBOP} and \ref{prop:preperiodicBOP} gives Theorem \ref{thm:babymainresult}.  Finally, Theorem \ref{thm:mainresult} is a consequence of combining  Theorem \ref{thm:babymainresult} and Corollary \ref{coro:zeromeasure}.

 \section{A detailed study of $L$ and $\tilde{f}$} \label{sec:Lf}

For a fixed nonlinear, monic polynomial $f \in \Z[z]$ of degree $d$ which is not conjugate to a power function or a Chebyshev polynomial, Theorem~\ref{thm:mainresult} characterizes the irreducible polynomials $P \in \Z[x, y]$ with $\m_f(P) = 0$ in terms of linear polynomials $L(z)$ and nonlinear polynomials $\tilde{f}(z)$ of minimal degree commuting with an iterate of $f$.  In this section we describe the polynomials $L$ and $\tilde{f}$ which arise in this manner. Having integral coefficients plays no special role here, so we will work throughout with polynomials over $\C$.

For a general given polynomial $f(z) = \sum_{i=0}^d a_i z^i$, we define $\beta = a_{d-1} / (da_d)$. In our case $f$ is monic, so $\beta = a_{d-1} / d$. The constant $\beta$ is invariant under replacing $f$ by one of its iterates. Conjugating $f$ by $z\mapsto z-\beta$ gives a polynomial $f_\beta$ of degree $d$ whose degree-$(d-1)$ coefficient is 0. Define $j(f)$ to be the greatest common divisor of those $i - 1$ for which $a_i \neq 0$. The key result is the following, due to Boyce:

\begin{thm}\cite[Theorem~4]{Boyce}
\label{boyce4}
If $g(z)$ is a degree-$m$ polynomial that commutes with $f(z)$, then the degree-$m$ polynomials commuting with $f(z)$ are exactly those of the form $ug(z) + (u - 1)\beta$ for $u$ a $j(f_\beta)$-th root of unity.
\end{thm}

This will allow us to completely understand the choices for both $L$ and $\tilde{f}$, modulo the calculation of $j(f^n_\beta)$, as soon as we have identified suitable seed polynomials $g$. In the linear case, we may take $g(z) = z$. The following result, 
a form of Ritt's theorem \cite{Ritt} as it appears in Boyce \cite{Boyce}, tells us how to proceed in the nonlinear case:

\begin{thm}\cite{Julia,Ritt} \label{thm:JR}
If two polynomials $f$ and $g$ commute under composition, then up to conjugation with the same linear polynomial, either both are power functions, both are Chebyshev polynomials, or both are roots of unity times iterates of the same polynomial (in which case the roots of unity themselves will commute with the latter polynomial).
\end{thm}

Since we are assuming that $f(z)$ is not conjugate to either $z^d$ or $\pm T_d(z)$, in the nonlinear case we may then take $g(z)$ to be a polynomial of minimal degree, some iterate of which is equal to $f(z)$ or to $f(z)$ times a root of unity that commutes with $g$. Note that such a polynomial will commute with all iterates of $f$, and will have minimal degree among all nonlinear polynomials which commute with any iterate of $f$.

We now wish to determine, for a fixed monic polynomial $f(z)$, what values are taken by the quantity $j(f^n_\beta)$. It will simplify the notation to index the coefficients of $f^n_\beta$ in order of decreasing degree, so we denote by $c_{n, i}$ the degree-$(d^n - i)$ coefficient of $f^n_\beta$; note that $c_{n, 0} = 1$ 
and $c_{n, 1} = 0$ for all $n$. Let $r$ denote the greatest common divisor of those indices $i$ for which $c_{1, i} \neq 0$. As $\gcd(d - 1, d - i - 1) = \gcd(d - 1, i)$ for any $i$, we have $j(f_\beta) = \gcd(d - 1, r)$.

Although $j(f^n_\beta)$ is defined in terms of the full set of coefficients of $f^n_\beta$, it turns out that it is in fact completely determined by its largest-degree coefficients $c_{n, 0}, \ldots, c_{n, d}$. We elaborate on this below.

\begin{lem}
For $i = 0, \ldots, d$, the degree-$(d^n - i)$ coefficient $c_{n, i}$ of $f^n_\beta(z)$ is equal to the degree-$(d^n - i)$ coefficient of $(f_\beta(z))^{d^{n-1}}$. In particular,
	\[ c_{n, i} = \sum_{i_1 + \cdots + i_{d^{n-1}} = i} c_{1, i_1} \cdots c_{1, i_{d^{n-1}}}. \]
\end{lem}

\begin{proof}
Notice that when computing $f^n_\beta(z)$ as $f_\beta(f^{n-1}_\beta(z))$ for $n \geq 2$, the contribution to the first $d + 1$ coefficients comes exclusively from the term $(f^{n-1}_\beta(z))^d$, since $f_\beta$ has no degree-$(d - 1)$ term, and the highest coefficient of $(f^{n-1}_\beta(z))^{d - 2}$ has degree $d^{n - 1}(d - 2) = d^n - 2d^{n-1} < d^n - d$.

Further, the first $d + 1$ coefficients of $(f^{n - 1}_\beta(z))^d$ are determined by the first $d + 1$ coefficients of $f^{n-1}_\beta(z)$, since any contribution of the coefficient of degree $d^{n-1} - d - 1$ or below will have at most degree $d^{n-1}(d - 1) + d^{n-1} - d - 1 = d^n - d - 1$.

Thus the first $d + 1$ coefficients of $f^n_\beta(z)$ are determined by the first $d + 1$ coefficients of $f^{n-1}_\beta(z)$, from the term of the form $(f^{n-1}_\beta(z))^n$. Proceeding by induction, we conclude that the first $d + 1$ coefficients of $f^n_\beta(z)$ are determined by the first $d + 1$ coefficients of $f_\beta(z)$ (i.e.\ all the coefficients of $f_\beta(z)$), from a term of the form $(f_\beta(z))^{d^{n-1}}$.

In other words, the first $d + 1$ coefficients of $f^n_\beta(z)$ are given by the first $d + 1$ coefficients of $(f_\beta(z))^{d^{n-1}}$.
\end{proof}

\begin{lem}
\label{j-geq}
Let $r = \gcd(i\, :\, c_{1, i} \neq 0)$, as above. Then $c_{n, i} = 0$ unless $i \equiv d^n \pmod{r}$.
\end{lem}
\begin{proof}
We have $c_{1, i} = 0$ unless $i \equiv d \pmod{r}$ by construction. Suppose towards induction that the statement holds for some fixed $n$. If one multiplies $kr + d$ monomials of degree equivalent to $d^n$ modulo $r$, the product has degree equivalent to $(kr + d)(d^n) \equiv d^{n+1} \pmod{r}$. The expression $(f^n_\beta(z))^{kr + d}$ is thus a sum of monomials whose degrees are equivalent to $d^{n+1}$ modulo $r$, and it follows that $f^{n+1}_\beta(z) = f_\beta(f^n_\beta(z))$ is also.
\end{proof}

\begin{lem}
\label{j-leq}
Enumerate the indices $i_0 < \ldots < i_m$ for which $c_{1, i_j} \neq 0$; note that $i_0 = 0$, and since $f$ is not conjugate to a power function, there is at least one nonzero coefficient $c_{1, i}$ with $i > 0$. Let
	\[ I = \{ i_k \, :\, \gcd(i_0, \ldots, i_k) \neq \gcd(i_0, \ldots, i_{k-1}) \}; \]
this set always contains at least $i_0$ and $i_1$. Then $c_{n, i} \neq 0$ for all $n$ and for all $i \in I$.
\end{lem}
\begin{proof}
For $k = 0, \ldots, m$, let $r_k = \gcd(i_0, \ldots, i_k)$, so that $I = \{i_k\, :\, r_k \neq r_{k-1}\}$.

As $i_1$ is the smallest positive index $i$ such that $c_{1, i} \neq 0$, it is also the smallest positive index such that $c_{n, i} \neq 0$, with $c_{n, i_1} = d^{n-1} c_{1, i_1}$ in this case.

Suppose that $i_k \in I$. It follows that $c_{n, i} = 0$ for $i < i_k$ with $i \not\equiv 0 \pmod{r_{k-1}}$. The coefficient $c_{n, i_k}$ is the sum of terms of the form
	\[ \prod_{s} c_{1, s}, \]
where the product runs over the entries of a $d^{n-1}$-tuple of indices summing to $i_k$. Since $r_k \neq r_{k-1}$, the integer $r_{k-1}$ does not divide $i_k$, so if $s < i_k$ for each index $s$, then at least one index has $s \not\equiv 0 \pmod{r_{k-1}}$, and the corresponding coefficient $c_{1, s}$ is thus zero. So the only nonzero terms have one index $s$ equal to $i_k$ and the others equal to 0; in other words,
	\[ c_{n, i_k} = d^{n-1} c_{1, i_k} \neq 0.\qedhere \]
\end{proof}

\begin{cor}
We have $j(f^n_\beta) = \gcd(d^n - 1, r)$.
\end{cor}
\begin{proof}
It follows from Lemma~\ref{j-geq} that $j(f^n_\beta) \geq \gcd(d^n - 1, r)$. From Lemma~\ref{j-leq}, we know that the degree-$(d^n - i)$ coefficient of $f^n_\beta(z)$ is nonzero for each $i \in I$, and as the greatest common divisor of the elements of $I$ is $r$, it follows that $j(f^n_\beta) \leq \gcd(d^n - 1, r)$ also.
\end{proof}

\begin{prop} \label{prop:ftildeL}
Suppose that $r = \gcd(i \, :\, c_{1, i} \neq 0)$ has prime factorization $\prod p_i^{e_i}$. Let
	\[ r' = \prod_{p_i \nmid d} p_i^{e_i}. \]
Then there are exactly $r'$ choices for the linear polynomial $L(z)$ and for the nonlinear polynomial $\tilde{f}(z)$ appearing in the statement of Theorem \ref{thm:mainresult}: the linear polynomials commuting with some iterate of $f$ are exactly those of the form $L(z) = uz + (u - 1)\beta$ for $u$ an $(r')$-th root of unity, and the nonlinear polynomials of minimal degree commuting with some iterate of $f$ are exactly those of the form $u\tilde{f}_0(z) + (u - 1)\beta$, where $\tilde{f}_0(z)$ is a polynomial of minimal degree, some iterate of which is equal to $f$, and $u$ is an $(r')$-th root of unity. 
\end{prop}
\begin{proof}
We have shown that $j(f^n_\beta) = \gcd(d^n - 1, r)$. If $p_i \mid d$, then $p_i \nmid d^n - 1$, so also $j(f^n_\beta) = \gcd(d^n - 1, r')$; in particular, $j(f^n_\beta) \mid r'$. As $\gcd(d, r') = 1$, $d$ is a unit in $\Z/r'\Z$, so there exists an $n$ such that $d^n \equiv 1 \pmod{r'}$, and in this case we have $j(f^n_\beta) = r'$. The result is now an immediate consequence of Theorem~\ref{boyce4}.
\end{proof}

\begin{exm}
Consider the polynomial $f(z) = f_\beta(z) = z^{14} - z^2$. Then $d = 14$, $\beta = 0$, $r = 12$, $r' = 3$, and
	\begin{align*}
	d^n &\equiv \begin{cases}
		2 \pmod{12} & \text{if $n = 1$}, \\
		4 \pmod{12} & \text{if $n$ is even}, \\
		8 \pmod{12} & \text{otherwise},
		\end{cases} \\
	&\equiv \begin{cases}
		2 \pmod{3} & \text{if $n$ is odd}, \\
		1 \pmod{3} & \text{if $n$ is even}.
		\end{cases}
	\end{align*}
Thus
	\[ j(f^n_\beta) = \gcd(d^n - 1, 3) = \begin{cases}
		1 & \text{if $n$ is odd}, \\
		3 & \text{if $n$ is even}.
		\end{cases} \]
Thus the linear polynomials commuting with some iterate of $f$ are exactly those of the form $L(z) = uz$ for $u$ a third root of unity. As $d$ is not a power of any smaller integer, $f$ is not an iterate of any polynomial of smaller degree, so the nonlinear polynomials of minimal degree commuting with an iterate of $f$ are exactly those of the form $\tilde{f}(z) = uf(z)$ for $u$ a third root of unity.
\end{exm}

\bibliographystyle{amsalpha}

\bibliography{Bibliography}

\end{document}